\documentclass[15pt]{amsart}
\usepackage{amsmath}
\usepackage{mathtools}
\usepackage{}
\usepackage{graphicx}
\usepackage[colorlinks=true, allcolors=blue]{hyperref}
\usepackage{amsfonts}
\usepackage{amsthm}
\usepackage{newlfont}
\usepackage{amscd}
\usepackage{amsgen}
\usepackage{amssymb}
\usepackage{mathrsfs}	
\usepackage{longtable}
\usepackage{listings}
\usepackage{extarrows}
\usepackage{tikz}
\usepackage{tikz-cd}
\usepackage{verbatim}
\numberwithin{equation}{section}
\usepackage[all]{xy}
\usepackage{color}
\usepackage{amssymb}
\usepackage[left=3cm,right=3cm]{geometry}
\usepackage{tikz}
\usepackage{tikz-cd}
\usepackage{mathtools}
\usepackage{bm} %bold mathcal
\usepackage{dynkin-diagrams}
\usetikzlibrary{matrix,shapes,arrows,decorations.pathmorphing}
\usepackage{calligra}
\usepackage{mathrsfs}
\usepackage{enumitem} % funny enumerate environments
\usepackage{tgpagella} %font
\usepackage[parfill]{parskip} %noindent everywhere
\usepackage{float}%override random table positioning
\restylefloat{table}
\usepackage{ytableau}
\usepackage{adjustbox}
\usetikzlibrary{calc,matrix,positioning}

\tikzset{ 
    table/.style={
        matrix of nodes,
        row sep=-\pgflinewidth,
        column sep=-\pgflinewidth,
        nodes={rectangle, text width=2.5em, text height = 1.5em, align=center},
        text depth=1.25ex,
        text height=2.5ex,
        nodes in empty cells
    },
}

\let\blb\mathbb

\def\CC{{\blb C}}

\def\LL{{\blb L}}

\def\NN{{\blb N}}

\def\PP{{\blb P}}

\def\SS{{\blb S}}

\def\VV{{\blb V}}

\def\ZZ{{\blb Z}}

\let\cal\mathcal

\def\Ec{{\cal E}}

\def\Hc{{\cal H}}
\def\Ic{{\cal I}}

\def\Mc{{\cal M}}
\def\Nc{{\cal N}}
\def\Oc{{\cal O}}

\def\Qc{{\cal Q}}

\def\Tc{{\cal T}}
\def\Uc{{\cal U}}
\def\Vc{{\cal V}}

\def\Xc{{\cal X}}

\DeclareMathOperator{\Gr}{Gr}
\DeclareMathOperator{\Fl}{Fl}

\DeclareMathOperator{\Sym}{Sym}

\newcommand{\grass}{\operatorname{Gr}}

\newcommand{\flag}{\operatorname{Fl}}

\newtheorem{lemma}{Lemma}[section]
\newtheorem{proposition}[lemma]{Proposition}
\newtheorem{theorem}[lemma]{Theorem}

\newtheorem{definition}[lemma]{Definition}

\newtheorem{todo}[lemma]{To do}

\theoremstyle{remark}

\newtheorem{remark}[lemma]{Remark}

\def\wt{\widetilde}

\def\arw{\longrightarrow}
\def\Hom{\operatorname{Hom}}
\def\Aut{\operatorname{Aut}}
\def\End{\operatorname{End}}
\def\im{\operatorname{im}}

\def\rk{\operatorname{rk}}

\def\Span{\operatorname{Span}}
\def\grass{\operatorname{Gr}}
\def\flag{\operatorname{Fl}}

\DeclareMathOperator{\sHom}{\mathscr{H}\text{\kern -3pt {\calligra\large om}}\,}

\newcommand\quotient[2]{
        \mathchoice
            {% \displaystyle
                \text{\raise1ex\hbox{$#1$}\Big/\lower1ex\hbox{$#2$}}%
            }
            {% \textstyle
                #1\,/\,#2
            }
            {% \scriptstyle
                #1\,/\,#2
            }
            {% \scriptscriptstyle  
                #1\,/\,#2
            }
    }

\makeatletter
\def\namedlabel#1#2{\begingroup
    #2%
    \def\@currentlabel{#2}%
    \phantomsection\label{#1}\endgroup
}
\makeatother

\title[New counterexamples to the birational Torelli theorem for Calabi--Yau manifolds]{New counterexamples to the birational Torelli theorem\\ for Calabi--Yau manifolds}

\author{Marco Rampazzo}
\address{
Alma Mater Studiorum Università di Bologna\\ Dipartimento di Matematica \\ Piazza di Porta San Donato 5\\ 40126 Bologna.}
\email[M.~ Rampazzo]{marco.rampazzo3@unibo.it, marco.rampazzo.90@gmail.com}

\begin{document}

\maketitle

\begin{abstract}
    We produce counterexamples to the birational Torelli theorem for Calabi--Yau manifolds in arbitrarily high dimension: this is done by exhibiting a series of non birational pairs of Calabi--Yau $(n^2-1)$-folds which, for $n \geq 2$ even, admit an isometry between their middle cohomologies. These varieties also satisfy an $\LL$-equivalence relation in the Grothendieck ring of varieties, i.e. the difference of their classes annihilates a power of the class of the affine line. We state this last property for a broader class of Calabi--Yau pairs, namely all those which are realized as pushforwards of a general $(1,1)$-section on a homogeneous roof in the sense of Kanemitsu, along its two extremal contractions.
\end{abstract}

\section{Introduction}

The global Torelli theorem asserts that the isomorphism class of a $K3$ surface is determined by its cohomological data: more precisely, two $K3$ surfaces are isomorphic if and only if their integral middle cohomologies, endowed with the intersection pairing, are isometric (see, for instance, \cite{huybrechts_k3s}). While similar statements have been formulated for hyperk\"ahler manifolds \cite{verbitsky_hks}, it is a natural question to ask whether a ``Torelli-type statement'' might exist for Calabi--Yau manifolds. In other words: does the existence of an isometry of middle cohomologies of two Calabi--Yau manifolds (a so-called Hodge-equivalence) imply that such manifolds are isomorphic, or birationally equivalent? The answer to this question has been proven to be negative, at least in low dimension: in \cite{ottemrennemo} and \cite{borisovcaldararuperry} a counterexample to the ``birational'' Torelli theorem has been found, in the form of a non-birational but Hodge-equivalent pair of Calabi--Yau threefolds. Shortly after, in \cite{manivel} a similar example has been given among Calabi--Yau fivefolds. However, to the author's knowledge, no further counterexamples have been found among Calabi--Yau manifolds of dimension higher than five.\\
\\
The main goal of this paper is to produce a series of new such counterexamples: namely, for any $n\in \NN$ greater than two, we construct pairs of Calabi--Yau $(n^2-1)$-folds which are not birationally equivalent, but they have isometric middle cohomology if $n$ is even. Such varieties are realized in the following way: let $V$ be a $(2n+1)$-dimensional complex projective space and $\flag := F(n, n+1, V)$ the variety parametrizing pairs of subspaces $W$, $W'$ of $V$, of dimension respectively $n$ and $n+1$, such that $W\subset W'$. This variety is one of the \emph{homogeneous roofs} classified by Kanemitsu \cite[Section 5.1.1]{kanemitsu}, i.e. it admits two projective bundle morphisms $p_-$ and $p_+$ of the same relative dimension, respectively over $\grass_-:= G(n, V)$ and $\grass_+ := G(n+1, V)$, and such that $\Oc(1,1):=p_-^*\Oc(1)\otimes p_+^*\Oc(1)$ is the Grothendieck line bundle of both the projective bundle structures.\\
As shown in \cite{ourpaper_k3s} and \cite{mypaper_roofbundles}, a general $s\in H^0(\Fl, \Oc(1,1))$ defines a pair of smooth Calabi--Yau varieties $(Y_-, Y_+):= (Z(p_{-*}s), Z(p_{+*}s))$. We show that such varieties are not isomorphic by proving that any isomorphism $Y_-\arw Y_+$ must descend from an isomorphism of the (unique) Grassmannians containing them, and ultimately from an automorphism $M$ of $H^0(\Fl, \Oc(1,1))$ satisfying $M S M^{-1} = S^T$, where $S$ is the matrix associated to $s$ as a $(1,1)$-divisor in the product of the Pl\"ucker spaces. By an argument similar to \cite{ottemrennemo, manivel} we verify that such $M$ cannot exist for the general $s$, thus proving that $Y_-$ and $Y_+$ are not isomorphic. Finally, we conclude that $Y_-$ and $Y_+$ cannot be birationally equivalent. This generalizes some of the results of \cite{ourpaper_cy3s} to higher dimension.\\
\\
The pairs we discuss are also ``$\LL$-equivalent'', i.e. the difference of their classes in the Grothendieck ring of varieties annihilates a power of the class of the affine line. We state this result for a wider class of pairs: by generalizing a result of Ito--Miura--Okawa--Ueda \cite{imou_G2}, we prove that every pair of Calabi--Yau varieties associated to a roof in the sense of \cite{ourpaper_k3s, mypaper_roofbundles} is $\LL$-equivalent. This property is conjectured to be related to derived equivalence \cite[Conjecture 1.6]{kuznetsovshinder}. The main results of this paper are gathered in the following theorem:

\begin{theorem}\label{thm:main_intro}(Theorem \ref{thm:main_body})
    For $n\in\NN$, $n\geq 2$, consider the locally trivial $\PP^n$-fibrations $p_-:F(n, n+1, 2n+1)  \arw G(n, 2n+1)$ and $p_+:F(n, n+1, 2n+1)  \arw G(n+1, 2n+1)$ and a general section $s\in H^0(F(n, n+1, 2n+1), p_-^*\Oc(1)\otimes p_+^*\Oc(1))$. Let $(Y_-, Y_+)$ be the pair of Calabi--Yau $(n^2-1)$-folds defined as $Y_\pm:=Z(p_{\pm*}s)$. Then:

        \begin{enumerate}
            \item $Y_-$ and $Y_+$ are not birationally equivalent
            \item $([Y_-]-[Y_+])\LL^n = 0$, i.e. $Y_-$ and $Y_+$ are $\LL$-equivalent
            \item For $n$ even, there is a Hodge isometry $H^{n^2-1}(Y_-, \ZZ) \simeq H^{n^2-1}(Y_+, \ZZ)$, i.e. $Y_-$ and $Y_+$ are Hodge equivalent.
        \end{enumerate}
\end{theorem}
    
In \cite[Section 5]{ottemrennemo} it is discussed how this construction for $n=2$ describes a divisor in the moduli space of intersections of general $PGL(10)$-translates of $G(2, 5)$. We show how this behavior does not extend to $n>2$.

\subsection*{Notations and conventions}
    We work over the field of complex numbers. %Given a vector space $V$, its projectivization is denoted following the convention $\PP(V) = \Proj(\Sym V)$. The same choice is adopted for a rank $r$ vector bundle $E$ on a variety $X$: we use the notation $\PP(E) = \Proj(\Sym E)$ for the associated locally trivial $\PP^{r-1}$-fibration over $X$.\\
    \\
    Let $\lambda = (\lambda_1, \dots, \lambda_r) \vdash n$ be an ordered partition of length $r$, i.e. a non-increasing list of $r$ integers such that $n = \sum_i\lambda_i$. Given a vector space $V$ of dimension $m\geq r$ (or a vector bundle $V$ of rank $m\geq r$) by $\SS^\lambda V$ we denote the Schur power of $V$ with respect to $\lambda$. In particular:
    
    \begin{enumerate}
        \item $\SS^{(l)} V = \Sym^l V$ for any $l\in\NN$, and $\SS^{(0)} V = \{0\}$
        \item $\SS^{(1^r)} V = \wedge^r V$ where $(1^r)$ is the partition of $r$ with exactly $r$ nonzero terms
        \item For $\lambda\vdash l$, $\nu\vdash n$ we write $\SS^\lambda V\otimes\SS^\nu V = \bigoplus_{\mu\vdash l+n} c_{\lambda\nu}^\mu\SS^\mu V$ where the integers $c_{\lambda\nu}^\mu$ are the Littlewood--Richardson coefficients
        \item By $\SS^\lambda\circ\SS^\mu V$ we denote the \emph{inner plethysm} (see \cite{weyman} for details).
    \end{enumerate}
    
    Let $G$ be a simple Lie group and $P\subset G$ a parabolic subgroup. Given the rational homogeneous variety $G/P$, we use the notation $\Ec_\omega = G\times^P \VV^P_{\omega}$ for the homogeneous, irreducible vector bundle of rank $r$ over $G/P$ associated with the $P$-representation of highest weight $\omega$, where $\VV^P_{\omega}$ is the rank $r$ vector space on which the $P$-representaton of highest weight $\omega$ acts. Similarly, we call $\VV^G_{\omega}$ the space on which the $G$-representation of highest weight $\omega$ acts.
    
\section*{Acknowledgments}
    The author wishes to thank Micha\l\ Kapustka and Giovanni Mongardi for reading the first draft of this work and providing valuable comments, and Francesco Denisi, Enrico Fatighenti and Stevell Muller for many helpful discussions. The author is also grateful to Sara A. Filippini, Jacopo Gandini and Jerzy Weyman for the advice given on the representation-theoretical aspects of Lemma \ref{lem:plethysm}. This work is supported by PRIN2017 ``2017YRA3LK'' and PRIN2020 ``2020KKWT53''. 
    
\section{Calabi--Yau pairs of type \texorpdfstring{$A^G_{2n}$}{something}}

    \subsection{Homogeneous vector bundles on Grassmannians}\label{subsec:grassmannians}
    Consider vector spaces $V_k$ and $V$, of dimensions respectively $k$ and $n$, where $k<n$. By $G(k, V)$ we denote the Grassmannian of $k$-dimensional subspaces of $V$, i.e. $G(k, V) = \{A\in\Hom(V_k, V) : \rk A = k\}/GL(V_k)$ where the quotient is taken with respect to the action $(g, A)\mapsto A g^{-1}$, for $g\in GL(V_k)$ and $A\in\Hom(V_k, V)$. This is a smooth variety of dimension $nk-k^2$, naturally embedded in $\PP(\wedge^k V)$ via the Pl\"ucker map $\psi_k:[A]\longmapsto[\psi_k(A)]$, where $\psi_k(A)$ is the vector of $k$-minors of $A$.\\
    The tautological vector bundle $\Uc$ on $G(k, V)$ is the rank $k$ bundle whose fiber over $[A]$ is the vector space $\im(A)\subset V$. It comes with an embedding into the rank $n$ trivial vector bundle, giving rise to the tautological short exact sequence on $G(k, V)$:

    \begin{equation}\label{eq:tautological_sequence}
        0\arw\Uc\arw V\otimes\Oc\arw\Qc\arw 0.
    \end{equation}

    One has $\wedge^k\Uc = \Oc(-1)$ and $\wedge^{n-k}\Qc = \Oc(1)$.

    We also recall that the Grassmannian is a rational homogeneous variety described as \linebreak $G(k, V)\simeq SL(V)/P_k$, where $P_k\subset SL(V)$ is an appropriate parabolic subgroup. Homogeneous, irreducible vector bundles are in one-to-one correspondence with representations of the Levi subgroup of $P_k$, which is $L_k = SL(V_k)\times \CC^*\times SL({V/V_k})\subset P_k$. We call $\omega_1, \dots, \omega_{k-1}$ the fundamental weights of the first factor, $\omega_k$ the one of the second factor and $\omega_{k+1}, 
    \dots, \omega_{n-1}$ the fundamental weights of the last block. In this way, if we consider two partitions $\lambda, \nu$ of appropriate lengths, one can check that:

    \begin{equation}\label{eq:homogeneous_bundles_grassmannian}
        \SS^\lambda\Uc^\vee\otimes\SS^\nu\Qc = \Ec_\mu
    \end{equation}
    
    where $\mu = \mu_1\omega_1+ \cdots + \mu_{m-1}\omega_{n-1}$ with coefficients given by
    
    \begin{equation}
        \mu_i = \left\{
            \begin{array}{cc}
                \lambda_i-\lambda_{i+1} & i<k \\
                \lambda_k + \nu_{n-k} & i = k \\
                \nu_{n-i} - \nu_{n-i+1} & i>k.
            \end{array}
            \right.
    \end{equation}

    All homogeneous, irreducuble vector bundles on $G(k, V)$ arise in this way. Their cohomology is usually computed via the Borel--Weil--Bott theorem, which we recall in Appendix \ref{appendix_bott}.

\subsection{Calabi--Yau varieties of type \texorpdfstring{$A^G_{2n}$}{something} and their properties}

Given a vector space $V$ of dimension $2n+1$, consider $\grass:= G(n, V)\subset \PP := \PP(\wedge^n V)$. By $\grass_g\subset \PP$ we mean the image of $\grass\subset\PP$ with respect to an isomorphism $g\in \Aut\PP = PGL(\wedge^n V)$ (we will often call such variety a \emph{translate} of $\grass$).

\begin{definition}\label{def:cy_type_AG2n}
    We say $Y\subset \grass$ is a Calabi--Yau variety of type $A_{2n}^G$ if $Y = Z(s)$ where $s\in H^0(\grass, \Qc^\vee(2))$ is a general section.
\end{definition}

\begin{remark}
    The name refers to \cite[Example 5.3]{kanemitsu}. In fact, as we will clarify further in this section, any $s$ as above is the pushforward of a hyperplane section on a so-called \emph{homogeneous roof}. These special Fano varieties of Picard number two are classified in \cite[Section 5.2.1]{kanemitsu}.
\end{remark}

One immediately computes $\dim Y = n^2-1$. The fact that $Y$ is indeed a smooth Calabi--Yau variety of Picard rank one is a consequence of \cite[Lemma 2.8]{mypaper_roofbundles} and \cite[Proposition 2.3]{ourpaper_generalizedroofs}.\\
The description of $Y$ as a smooth, transverse zero locus of a section gives rise to a Koszul exact sequence on $\grass$:

\begin{equation}\label{eq:koszul}
    0\arw \wedge^{n+1}\Qc(-2n-2)\arw \cdots \arw \wedge^l\Qc(-2l)\arw \cdots \arw \Qc(-2)\arw\Ic_{Y|\grass}\arw 0.
\end{equation}

\begin{lemma}\label{lem:one_section}
    Let $Y\subset \grass$ be a Calabi--Yau variety of type $A^G_{2n}$ such that $Y = Z(s) = Z(s')$, where $s$ and $s'$ are sections of $\Qc^\vee(2)$. Then $s = \lambda s'$ for some $\lambda\in\CC^*$.
\end{lemma}

\begin{proof}
    Both $s$ and $s'$ give rise to resolutions as in Equation \ref{eq:koszul}, and thus to the diagram:
    
    \begin{equation}
        \begin{tikzcd}
            \cdots\ar{r} & \Qc(-2)\ar{d}{\beta}\ar{r}{f_s} & \Ic_{Y|\grass}\ar[equals]{d}\ar{r} & 0 \\
            \cdots\ar{r} & \Qc(-2) \ar{r}{f_{s'}} & \Ic_{Y|\grass}\ar{r} & 0
        \end{tikzcd}
    \end{equation}
    
    Let us first verify that $\beta$ exists, by checking that $\Hom_{\grass}(\Qc(-2), \Qc(-2))\xrightarrow{\,\,\alpha\,\,} \Hom_{\grass}(\Qc(-2), \Ic_{Y|\grass})$ is surjective. To this purpose, if we take the tensor product of $\Qc^\vee(2)$ with Equation \ref{eq:koszul} we find

    \begin{equation}\label{eq:tensored_koszul_1}
        \begin{tikzcd}[column sep = small]
            \cdots\ar{rr} & &  \Qc^\vee(2)\otimes\wedge^{2}\Qc(-4)\ar{rr}\ar[two heads]{dr} & &  \Qc^\vee(2)\otimes\Qc(-2) \ar{rr}{\alpha} & & \Qc^\vee(2)\otimes\Ic_{Y|\grass} \ar{rr} & &  0.\\
            & & & \ker{\alpha} \ar[hook]{ur} & & & & &
        \end{tikzcd}
    \end{equation}
    
    Taking the long exact sequence of cohomology of the last short exact sequence we get:

    \begin{equation*}
        0\arw H^0(\grass, \ker\alpha)\arw \Hom_{\grass}(\Qc(-2), \Qc(-2)) \xrightarrow{\,\,\alpha\,\,} \Hom_{\grass}(\Qc(-2), \Ic_{Y|G}) \arw H^1(\grass, \ker\alpha)
    \end{equation*}
    
    hence $\alpha$ is surjective iff $H^1(\grass, \ker\alpha) = 0$. By Equation \ref{eq:tensored_koszul_1} we need to check that $H^p(\grass, \Qc^\vee(2)\otimes\wedge^k\Qc(-2k)) = 0$ for $0\leq k \leq n+1$ and $p<n+1$, which is computed in Lemma \ref{lem:vanishings_hom_is_surjectve}. This settles the problem of the existence of $\beta$.\\
    The identity $I\in\Aut(\Ic_{Y|\grass})$ lifts to an automorphism of $\Qc(-2)$, but the latter is slope-stable \cite[Theorem 2.4]{umemura_stability} and exceptional \cite{kapranov_grassmannians}: this implies that its only endomorphisms are multiples of the identity. Therefore $s = \lambda s'$.
\end{proof}

    The purpose of the next results (Lemma \ref{lem:q_is_stable}, Lemma \ref{lem:maps_between_qs} and Proposition \ref{prop:unique_grassmannian}) is to show that each Calabi--Yau manifold of type $A^G_{2n}$ is contained in a unique translate of $\grass$.

\begin{lemma}\label{lem:q_is_stable}
    Let $Y \subset \grass$ be a Calabi--Yau variety of type $A_{2n}^G$. Then $\Qc|_Y$ is slope-stable.
\end{lemma}

\begin{proof}
    We use Hoppe's criterion \cite[Proposition 1]{jardimmenetprataearp}: we need to check that the bundle $\wedge^l\Qc(-1)|_Y$ has no sections for $1\leq l \leq n$.\\
    In light of the tensor product of the Koszul resolution for $Y\subset \grass$ (Equation \ref{eq:koszul}) with $\wedge^l\Qc(-1)$, we must ensure that for every $l\leq n+1$ one has $H^p(\grass, \wedge^k\Qc\otimes\wedge^l\Qc(-1-2l)) = 0$ for any $p<n+1$. This is done in Lemma \ref{lem:q_is_stable_vanishings}.
\end{proof}

%{\color{green} The formulation there exists g doesn't fix g}

\begin{lemma}\label{lem:maps_between_qs}
    Let $Y = Z(s)\subset \grass$ be a Calabi--Yau variety of type $A_{2n}^G$, and let $g\in\Aut(\PP)$ be such that $Y = Z(\wt s) \subset \grass_g$ for some $\wt s\in H^0(\grass_g, \Qc^\vee_{\grass_g}(2))$. Then $\Qc_{\grass}|_Y\simeq \Qc_{\grass_g}|_Y$.
\end{lemma}

\begin{proof}
    Recall that by Lemma \ref{lem:q_is_stable} $\Qc_{\grass}|_Y$ is stable, and therefore also $\Qc_{\grass_g}|_Y$ is stable. This implies that any morphism between them is either identically zero or an isomorphism: to prove that such bundles are isomorphic, we just need to show a nonzero morphism between them.\\
    By applying the functor $\Hom_Y(-, \Qc_{\grass_g}^\vee(2))$ to the normal bundle sequence of $Y\subset \grass$ and taking the long exact sequence of cohomology, one finds:

    \begin{equation*}
        \begin{tikzcd}
            0\ar{r} & \Hom_Y(\Qc_{\grass}^\vee, \Qc^\vee_{\grass_g})\ar{r} & \Hom_Y(T_{\grass}|_Y, \Qc^\vee_{\grass_g}(2))\ar{r}{\beta} & \Hom_Y(T_Y, \Qc^\vee_{\grass_g}(2)).
        \end{tikzcd}
    \end{equation*}

    Thus, our statement is proven once we show that $\beta$ is not injective.\\
    Given the embedding $\iota:Y\xhookrightarrow{\,\,\,\,\,\,\,}\grass$ and the differential $d\iota:T_Y\xhookrightarrow{\,\,\,\,\,\,\,} T_{\grass}|_Y$ (and the similarly defined maps $\wt\iota$ and $d\wt\iota$), one has $\beta(f) = f\circ d\iota$. We will explicitly construct a map $K_g$ which makes the following diagram commutative:

    \begin{equation}\label{eq:triangle_diagram}
        \begin{tikzcd}[row sep = huge, column sep = huge]
            T_Y \ar[hook]{r}{d\iota}\ar[hook, swap]{dr}{d\wt \iota} & T_{\grass}|_Y\ar{d}{K_g} \\
            & T_{\grass_g}|_Y.
        \end{tikzcd}
    \end{equation}

    For every $h\in \Aut\PP$, to any point of $\grass_h$ corresponds a $n$-form, which is totally decomposable in the basis of $\wedge^n V$ where $h$ is the identity. Since $Y\subseteq \grass\cap \grass_g$, to every $x\in Y$ correspond two forms: one is $y = \iota(x)$ and the other is $\wt y = \wt\iota(x)$. In particular, once we fix a basis of $V$ (which induces a basis for $\wedge^n V$), if $y = y_1\wedge\dots\wedge y_n$ it follows that $\wt y$ is still a point of $\grass$, and therefore $\wt y$ is again decomposable in the same basis of $\wedge^n V$. Hence, since $GL(V)$ acts transitively on $\grass$, we can always set $\wt y = F_{g, y} y_1\wedge \dots \wedge F_{g, y} y_n$, where $F_{g, y}\in GL(V)$ is a map sending $\Span(y)$ to $\Span(\wt y)$ and $V/\Span(y)$ to $V/\Span(\wt y)$, which depends on $y$ and $g$.\\
    \\
    Recall that $T_{\grass}|_Y\simeq \Hc om_Y(\Uc_{\grass}, \Qc_{\grass})$. Then, we define $K_g$ in the following way: first, given a point $(y, f)$ in (a local trivialization of) the total space of $T_{\grass}|_Y$, where $f:\Span(y)\arw V/\Span(y)$, we set

    \begin{equation*}
        K_{g, y}: (y, f) \longmapsto (\wt y, \wt f):=(\wt y, F_{g, y}|_{V/\Span(y)}\circ f\circ F_{g, y}^{-1}|_{\Span(\wt y)}), 
    \end{equation*}

    then we observe that this map globalizes to an isomorphism of vector bundles $K_g:T_G|_Y\arw T_{\grass_g}|_Y$, since it is an isomorphism for every $y$ and it commutes with the change of local chart (which in the description of $\Gr$ as a $GL(V_n)$-quotient of $\Hom(V_n, V)$ simply amounts to change of basis). To give a global description of $K_g$, observe that $g$ induces an  automorphism $\kappa_g\in\Aut(Y)$ which exchanges the two Grassmannian translates containing $Y$. Then $\kappa_g$ acts on $n$-forms sending $\wt y$ to $y$, and by taking the pullback of $T_{\grass}|_Y$ we find $K_g = \kappa_g|_Y^{\,\,\,*}$.\\
    Diagram \ref{eq:triangle_diagram} commutes by construction.
    Combining the normal bundle sequences associated to the two descriptions of $Y$ as a zero locus, one has:

    \begin{equation*}
        \begin{tikzcd}[row sep = huge, column sep = huge]
            T_Y \ar[swap, equals]{d}{I} \ar[hook]{r}{d\iota} & T_{\grass}|_Y \ar[swap]{d}{K_g}\ar[two heads]{r}{\tau} & \Qc_{\grass}^\vee|_Y(2) \\
            T_Y \ar[hook]{r}{d\wt\iota} & T_{\grass_g}|_Y \ar[two heads]{r}{\wt \tau} & \Qc_{\grass_g}^\vee|_Y(2). \\
        \end{tikzcd}
    \end{equation*}
    
    The goal now is to find $\phi\in\Hom_Y(T_G|_Y, \Qc_{\grass_g}^\vee(2))$ nonzero such that $\beta(\phi) = 0$. Choose $\phi := \wt\tau\circ K_g$. This map is not identically zero: to convince ourselves of this, we just need to take $(y, f)$ such that  $(\wt y, \wt f)$ lies in the preimage of a point $(\wt y, v)\in\Qc_{\grass_g}^\vee(2)$, with $v\neq 0$. Then $\wt\tau\circ K_g (y, f) = \wt\tau(\wt y, \wt f) \neq 0$.
    %since $K_g\circ d\iota = d\wt\iota$ and $(y,f)\notin\im d\iota$, one has $(\wt y, \wt f)\notin\im d\wt\iota = \ker\wt\tau$.

    Summing all up, we have $\beta(\phi) = \wt\tau\circ K_g \circ d\iota = \wt\tau \circ d\wt\iota = 0$ and therefore $\Qc_{\grass}|_Y\simeq\Qc_{\grass_g}|_Y$.
\end{proof}

\begin{proposition}\label{prop:unique_grassmannian}
    Let $Y = Z(s)\subset\grass$ be a Calabi--Yau variety of type $A_{2n}^G$. Then there exists no nontrivial $g\in\Aut(\PP)$ such that $Y = Z(\wt s) \subset \grass_g$ for $\wt s\in H^0(\grass_g, \Qc^\vee_{\grass_g}(2))$.
\end{proposition}

\begin{proof}
    By contradiction, pick $g\in\Aut(\PP)$ as above. By \cite[Proposition 2.1]{arrondo}, the embedding of a subvariety $Y$ in $\grass_g$ is determined by the isomorphism class of $\Qc_{\grass_g}|_Y$. We will now proceed essentially as in \cite[proof of Lemma 2.2]{borisovcaldararuperry}: observe that $V \simeq H^0(\grass_g, \Qc_{\grass_g})$, and that the Pl\"ucker embedding of $\grass_g$ is given by the isomorphism $\wedge^{n+1} H^0(\grass_g, \Qc_{\grass_g})^\vee \simeq H^0(\PP, \Oc(1))$. Since by Lemma \ref{lem:spaces_of_sections} one has $H^0(Y, \Qc_{\grass_g})\simeq H^0(\grass_g, \Qc_{\grass_g})$ and $H^0(Y, \Oc(1)) \simeq H^0(\PP, \Oc(1))$, the isomorphism class of $\Qc_{\grass_g}|_Y$ determines also the Pl\"ucker embedding $\grass_g\xhookrightarrow{\,\,\,\,\,}\PP$, and thus the isomorphism $g$. Since by \ref{lem:maps_between_qs} all restrictions of quotient bundles to $Y$ are isomorphic, we conclude that there exists a unique Grassmannian translate $\grass$ containing $Y$ as a zero locus of $\Qc_{\grass}^\vee(2)$.
\end{proof}

%\section{Non isomorphic Calabi--Yau pairs}

\subsection{The main construction}
Let us consider the flag variety $\flag:= F(n, n+1, V)$. It is a homogeneous roof in the sense of \cite{kanemitsu, ourpaper_k3s}, i.e. it admits two surjections over rational homogeneous varieties, and these maps are projectivizations of homogeneous vector bundles with the same Grothendieck line bundle. From now on, let us call $\grass_-:= G(n, V)$ and $\grass_+ = G(n+1, V)$, with their Pl\"ucker embeddings $\grass_\pm\xhookrightarrow{\,\,\,\,\,\,}\PP_\pm$, and call $\Uc_\pm, \Qc_\pm$ the respective tautological and quotient bundles. Then, $\flag = \PP(\Qc_-^\vee(2)\arw \grass_-) = \PP(\Uc_+(2)\arw \grass_+)$, giving rise to the ``roof diagram'':

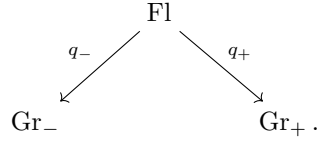
\begin{equation*}
    \begin{tikzcd}[row sep = large]
        & \flag\ar[swap]{dl}{q_-}\ar{dr}{q_+} & \\
        \grass_- && \grass_+.
    \end{tikzcd}
\end{equation*}

The following is a special case of the more general definition of a Calabi--Yau pair associated to a roof \cite{ourpaper_k3s, mypaper_roofbundles}:

\begin{definition}
    Let $s\in H^0(\flag, \Oc(1,1))$. be general. Then we call $(Y_-, Y_+)$ a Calabi--Yau pair of type $A^G_{2n}$, where $Y_\pm:= Z(p_{\pm *} s) \subset \grass_\pm$. 
\end{definition}

While the properties of $Y_-$ have been described in the previous sections, note that $\grass_-\simeq \grass_+$, and that such isomorphism sends $\Qc_+^\vee(2)$ to $\Uc_-(2)$: hence, everything we said about $Y_-$ applies to $Y_+$ verbatim. In particular, $Y_\pm$ is contained as a zero locus in a unique translate of $\grass_\pm$, and its isomorphism class determines the section $s_\pm:= q_{\pm *} s$ up to scalar multiplication.

\begin{remark}
    The Calabi--Yau pairs of type $A^G_{2n}$ can also be described by means of a mathematical-physical construction: they appear as vacuum manifolds of a suitably constructed gauged linear sigma model \cite{ourpaper_generalizedroofs}. In a more mathematical parlance, for $n>1$, one can construct, through a variation of GIT, a birational map between two total spaces of vector bundles of rank $n+1$, respectively on $\grass_-$ and $\grass_+$, and the critical loci of a superpotential restricted to such total spaces are isomorphic to $Y_-$ and $Y_+$.
\end{remark}

The following statement is the direct extension of \cite[Corollary 2.3]{ourpaper_cy3s} to $n>2$. In particular, the proof is identical and will therefore be omitted.

\begin{lemma}\label{lem:one_section_above}
    Consider $Y_\pm = Z(q_{\pm*}s)\subset \grass_\pm$ as above. Then there exists a unique $s\subset H^0(\flag, \Oc(1,1))$, up to rescaling, such that given $q_\pm: Z(s)\arw \grass_\pm$ one has:
    
    \begin{equation}
        q_\pm^{-1}(y) \simeq
        \left\{
        \begin{array}{ll}
            \PP^{n} & y\in Y_\pm \\
            \PP^{n-1} & y\in \grass_\pm\setminus Y_\pm.
        \end{array}
        \right.
    \end{equation}
\end{lemma}

\subsection{Automorphisms of the space of sections}
The remainder of this section is again a direct generalization to $n>2$ of some results of \cite{ourpaper_cy3s}, hence we will be brief.\\
Let us observe that any isomorphism $f:\grass_-\arw \grass_+$ is induced by the choice of an isomorphism $T_f:V\arw V^\vee$. In fact, given the canonical isomorphism $D:G(n, V^\vee)\arw G(n+1, V) = \grass_+$ induced by the outer automorphism of the Dynkin diagram $A_{2n}$, one has $f = D\circ \tau_-(f)$ where $\tau_-(f)$ is the action of $T_f$ on $\grass_-$, i.e. $\tau_-(f):\grass_-=G(n, V)\arw G(n, V^\vee)$. All these maps extend to $\PP_{-}$, and therefore we define the automorphism $\iota_f\in\Aut(\PP_-\times\PP_+)$ by setting $\iota_f (x, y) = ((f^\vee)^{-1}(y), f(x))$. This map lifts to an automorphism $\wt\iota_f$ of $H^0(\PP_{-}\times\PP_{+}, \Oc(1,1))$ by setting $\wt\iota_f(s) := s\circ\iota_f$.

Let us now describe how $\wt\iota_f$ acts explicitly. The data of a section $s\in H^0(\PP_{-}\times\PP_{+}, \Oc(1,1))$ is contained in a matrix $S\in\End(\wedge^nV)$ in the following way:

\begin{equation}
    \begin{tikzcd}[row sep = tiny, column sep = huge, /tikz/column 1/.append style={anchor=base east} ,/tikz/column 3/.append style={anchor=base west}]
        \PP_{-}\times\PP_{+} \ar{r}{s} & \Oc(1,1) \\
        \displaystyle{[x], [y]} \ar[maps to]{r} & \displaystyle{[x, y, y^T S  x]}
    \end{tikzcd}
\end{equation}

where, naturally, one has $[x\lambda, y\lambda', \lambda\lambda'y^T S  x] = [x, y, y^T S  x]$. From now on, we will commit the slight abuse of notation of identifying sections $s\in H^0(\PP_-\times\PP_+, \Oc(1,1)$ with the associated matrices $S\in\End(\wedge^nV)$.

\begin{lemma}\label{lem:transposition}
    Consider $f = D\circ\tau_{-}(f)$, a section  $S$ of $\Oc_{\PP_-\times\PP_+}(1,1)$ and the matrix $M_f$ such that $f(x) = M_f x$. Then $\wt\iota_f(s)= M_f^{-1} S ^T M_f$.
\end{lemma}

\begin{proof}
    By definition:
    \begin{equation*}
        \begin{split}
            \wt\iota_f(s)([x], [y]) & = s\circ\iota_f([x], [y])\\
                                    & = s([(f^\vee)^{-1}(y)], [f(x)])\\
                                    & = [(f^\vee)^{-1}(y), f(x), f(x)^T S  (f^\vee)^{-1}(y)]
        \end{split}
    \end{equation*}
    
    Since $(f^\vee)^{-1}(y) = M_f^{-T} y$ we conclude by
    
    \begin{equation}
        f(x)^T S  (f^\vee)^{-1}(y) = (M_f x)^T S  M_f^{-T} y = y^T M_f^{-1} S ^T M_f x.
    \end{equation}
    
\end{proof}

\begin{lemma}\label{lem:fix_section_f_duality}
    Consider $Y_{-} = Z(p_{-*}S)$ and $\wt Y_{-}=Z(p_{-*}\wt S)$ smooth, where $S, \wt S\in H^0(\flag, \Oc(1,1))$, and let $f:\grass_-\arw \grass_+$ be an isomorphism. Then $(Y_{-}, f(\wt Y_{-}))$ is a Calabi--Yau pair of type $A^G_{2n}$ if and only if there exists $\lambda\in\CC^*$ such that $\wt S = \lambda S$
\end{lemma}

\begin{proof}
    The proof is identical to the one of \cite[Lemma 7.1.7]{ourpaper_cy3s}, hence it will be omitted. In particular, note that the the proof of \cite[Lemma 7.1.7]{ourpaper_cy3s} depends on \cite[Lemma 7.1.5]{ourpaper_cy3s}, but the first proof of the latter applies to our case verbatim.
\end{proof}

\section{\texorpdfstring{$\LL$}{something} -equivalence, Hodge equivalence and non birationality}

    \subsection{Non birational Calabi--Yau pairs}\label{subsec:nonbirational}

    Let $(Y_-, Y_+)$ be a Calabi--Yau pair of type $A^G_{2n}$ and consider a birational equivalence $f:Y_-\dashrightarrow Y_+$. There is a standard argument to show that $Y_-$ and $Y_+$ must be isomorphic (see, for instance, \cite[Proposition 4.4]{manivel}). We summarize it here. Since $Y_-$ and $Y_+$ are Calabi--Yau, they are minimal models, and this implies that $f$ must be an isomorphism out of codimension two. However our varieties have Picard number one, and therefore $f$ identifies the spaces of sections of the respective very ample generators of the Picard groups, yielding a projective equivalence.\\
    \\
    The main result of this section is Proposition \ref{prop:notiso}: it states that no isomorphism $f:Y_-\arw Y_+$ can exist, and thus, by the discussion above, $Y_-$ and $Y_+$ cannot be birationally equivalent. We will begin by proving the following lemma, which is necessary for Proposition \ref{prop:notiso}.

    \begin{lemma}\label{lem:plethysm}
        There exists a representation $\SS^\lambda V$ such that in the decomposition of the plethysm $\SS^\lambda\circ\SS^{(1^n)}V$ there is a summand of the form $\SS^{(k^{2n+1})}V$ with multipilcity higher than one.
    \end{lemma}

    \begin{proof}
        As observed in \cite[proof of Proposition 4.4]{manivel}, if there were no such $\lambda$, then there would be a dense orbit of the action of $GL(V)$ on the complete flag variety $F(1,2,\dots, \wedge^n V)$. We rule out this possibility by a simple dimension count. Fix $N = \dim\wedge^n V$. One has a tower of surjections:
        \begin{equation*}
            \begin{tikzcd}
                F(1,\dots, \wedge^n V)\ar{d}{f_1}\\                    F(2,\dots, \wedge^n V)\ar{d}{f_2}\\
                \vdots\ar{d}{f_{N-3}} \\
                F(N-2, N-1, \wedge^n V)\ar{d}{f_{N-2}}\\
                G(N-1, \wedge^n V)
            \end{tikzcd}
        \end{equation*}
        where $f_k$ is a $\PP^k$-bundle for every $k$. It follows that:
        \begin{equation*}
            \dim F(1,\dots, \wedge^n V) = 1+2+\cdots + N-1 = \frac{N(N-1)}{2} = \frac{1}{2}\binom{2n+1}{n}\left[\binom{2n+1}{n}-1\right]
        \end{equation*}
        Since $GL(V)$ has dimension $(2n+1)^2$, which for $n>1$ is strictly smaller then the number above, the $GL(V)$-orbit of any point of $F(1,\dots, \wedge^n V)$ cannot be dense. 
    \end{proof}
    
    \begin{proposition}\label{prop:notiso}
        Let $S\subset H^0(\flag, \Oc(1,1))$ be general. Then the associated pair of Calabi--Yau varieties are not birationally equivalent.
    \end{proposition}

    \begin{proof}
        We can assume that the class of equivalence of a general $S$ up to scalar multiplication is an element of $PGL(\wedge^n V)$. By Lemma \ref{lem:fix_section_f_duality}, an isomorphism $f:Y_-\arw Y_+$ would be such that $\wt\iota_f(S) = S$ up to rescaling. Hence, by Lemma \ref{lem:transposition}, proving that there is no such isomorphism amounts to find $S$ such that there is no $M$ satisfying $SM = MS^T$, where $M$ is chosen among (extensions to $\PP_{-}$ of) of isomorphisms $\grass_-\arw \grass_+$. Fix $G = PGL(V)\times PGL(V)$. Following the approach of \cite[proof of Lemma 4.7]{ottemrennemo}, the proposition is proven if we verify a stronger statement, i.e. that the equivalence class of $S$ in $PGL(\wedge^2 V)$ is not in the orbit of the one of $S^T$ under the following $G$-action:

        \begin{equation*}
            \begin{tikzcd}[row sep = tiny, column sep = huge, /tikz/column 1/.append style={anchor=base east} ,/tikz/column 3/.append style={anchor=base west}]
                PGL(\wedge^2 V) \times G \ar{r} & PGL(\wedge^2 V)\\
                \displaystyle{[S], [M], [N]}\ar[maps to]{r} & \displaystyle{[\psi(M)^{-1}S\psi(N)]}
            \end{tikzcd}
        \end{equation*}

        where $\psi(M)$ is the matrix of minors of order $n$ of $M$, and the same for $N$. Let us first prove that the general $S\in SL(\wedge^2 V)$ is not in the same $\wt G$-orbit of $S^T$ where $\wt G:=SL(V)\times SL(V)$, since the argument is easier and it can be easily adapted to prove the main result afterwards. To prove this simpler statement is enough to verify the following:
        
        \textbf{Claim.} There exists a function $h\in \CC[SL(\wedge^n V)]^{\wt G}$ which is not preserved by the involution $\tau:g\longmapsto g^t$ of $SL(\wedge^n V)$.
        
        First note that
        
        \begin{equation*}
            \CC[SL(\wedge^n V)]^{\wt G} = \CC[\wedge^n V^\vee\otimes \wedge^n V]/(\det - 1)
        \end{equation*}

        and that there is a decomposition

        \begin{equation*}
            \begin{split}
                 \CC[\wedge^n V^\vee\otimes \wedge^n V] &= \bigoplus_k \Sym^k(\wedge^n V^\vee\otimes \wedge^n V) \\
                 & = \bigoplus_k\bigoplus_{\lambda\vdash k}\SS^\lambda\wedge^nV^\vee\otimes\SS^\lambda\wedge^nV.
            \end{split}
        \end{equation*}

        If we now take $\alpha\in \wedge^n V^\vee$, $\beta\in\wedge^n V$ and the isomorphism $T:\wedge^n V\arw \wedge^n V^\vee$ defined by a choice of a basis of $V$, we immediately see that for every $\lambda$ the action of $\tau$ is

        \begin{equation*}
            \tau: \SS^\lambda\alpha\otimes\SS^\lambda\beta\longmapsto \SS^\lambda T^{-1}(\beta)\otimes\SS^\lambda T(\alpha)
        \end{equation*}

        By Lemma \ref{lem:plethysm} we can find a partition $\lambda$ such that $\SS^\lambda\circ\SS^{(1^n)}V$, as a sum of irreducible $GL(V)$-representations, contains a one-dimensional summand of multiplicity higher than one: this implies that $\dim (\SS^\lambda \wedge^n V)^{SL(V)} > 1$. Hence, choosing a general pair $\alpha\in \wedge^n V^\vee$, $\beta\in\wedge^n V$, the function $\SS^\lambda \alpha\otimes \SS^\lambda \beta\in\CC[\wedge^n V^\vee\otimes\wedge^n V]^{\wt G}$ is not fixed by $\tau$, and the same holds for its image $h$ through the quotient map $\CC[\wedge^n V^\vee\otimes\wedge^n V]^{\wt G}\arw \CC[SL(\wedge^n V)]^{\wt G}$. This settles the claim.\\
        Observe now that choosing $h^2$ instead of $h$ gives rise to a $G$-invariant function which descends to the quotient $\CC[PGL(\wedge^2 V)]^G$, and which is not fixed by $\tau$.\\
        \\
        This proves that $Y_-$ and $Y_+$ are not isomorphic, hence they are not projectvely equivalent. By the discussion above (beginning of Section \ref{subsec:nonbirational}), we conclude that they are also not birationally equivalent.
    \end{proof}
    
    \subsection{Every roof leads to \texorpdfstring{$\LL$}{something} -equivalent pairs}
    Let us, for a moment, work in the more general context of homogeneous roofs in the sense of \cite[Section 5.1.1]{kanemitsu}. Fix a homogeneous roof $X = G/P$, with $G$ simply connected, simple Lie group and $P\subset G$ parabolic. It comes with two projective bundle structures $X \simeq \PP(\Ec_\pm\arw X_\pm)$ where $X_\pm = G/P_\pm$ is a Picard rank one rational homogeneous variety, with $P_\pm\subset P$. Call $p_\pm$ the structure maps of the projective bundle morphisms.\\
    Given $\Oc(1,1) = \Oc(1)\boxtimes\Oc(1)$ on $X$, one has $p_{\pm*}\Oc(1,1) = \Ec_\pm$. For a general section $s\in H^0(X, \Oc(1,1)$ define $M:=Z(S)\subset X$ and $Y_\pm:=Z(p_{\pm*}s)\subset X_\pm$. Then, by restricting $p_\pm$ to $M$ we find morphisms:

    \begin{equation}\label{eq:roof_general}
        \begin{tikzcd}
            & M\ar[swap]{dl}{q_-}\ar{dr}{q_+} & \\
            X_- & & X_+.
        \end{tikzcd}
    \end{equation}

    As observed in \cite[Section 2.1]{ourpaper_k3s} these maps are \emph{piecewise trivial} projective bundle fibrations and their fibers are:

    \begin{equation}
        q_\pm^{-1}(y) \simeq
        \left\{
        \begin{array}{ll}
            \PP^{r-1} & y\in Y_\pm \\
            \PP^{r-2} & y\in X_\pm\setminus Y_\pm
        \end{array}
        \right.
    \end{equation}

    where $r = \rk\Ec_\pm$. This leads to the following identities in the Grothendieck ring of varieties:

    \begin{align}
        [X] &= [X_\pm](1+\LL+\cdots+\LL^{r-1}) \label{eq:Leq_projbundle} \\
        [M] &= [X_\pm](1+\LL+\cdots+\LL^{r-2}) + [Y_\pm]\LL^{r-1} \label{eq:Leq_hyperplane}
    \end{align}

    where we used the well-known identity $[\PP^{m}] = 1+\LL+\cdots+\LL^{m}$. 
    
    \begin{proposition}\label{prop:Leq}
        Let $(Y_-, Y_+)$ be a Calabi--Yau pair associated to a homogeneous roof $X$ as above. Then one has:

        \begin{equation}\label{eq:L_equivalence_general}
            ([Y_-]-[Y_+])\LL^{r-1} = 0.
        \end{equation}
    \end{proposition}

    \begin{proof}
        By Equation \ref{eq:Leq_hyperplane} we immediately write (cf. \cite[Equation 2.2]{ourpaper_k3s}):

        \begin{equation}\label{eq:L_equivalence_intermediate_step}
            ([X_-] - [X_+])(1+\LL+\cdots+\LL^{r-2}) + ([Y_-] - [Y_+])\LL^{r-1} = 0
        \end{equation}

        and therefore we are able to conclude if we can show that $[X_-] = [X_+]$. This last equality is essentially \cite[proof of Proposition 2.3]{imou_G2}: the authors address the single case $X = G_2/B$, but their argument is general enough to be applied in the present context without modification.
    \end{proof}

    \begin{remark}
        Note that the argument of \cite{imou_G2} is not necessary to prove that Equation \ref{eq:L_equivalence_general} holds for Calabi--Yau pairs of type $A^G_{2n}$: in fact, in this case $X_\pm = \grass_\pm$, and the first summand of Equation \ref{eq:L_equivalence_intermediate_step} vanishes for the simple reason that $\grass_-$ and $\grass_+$ are isomorphic.
        \end{remark}
        
    \subsection{Hodge equivalence for the odd-dimensional pairs}

    Recall the following standard terminology:
    
    \begin{definition}
        We say that two varieties $Y_-$ and $Y_+$ of dimension $d$ are Hodge-equivalent if there exists an isometry $H^d(Y_-, \ZZ)\arw H^d(Y_+, \ZZ)$ with respect to the cup product.
    \end{definition}

    Let us come back to the setting of  Diagram \ref{eq:roof_general}, where we set $d = \dim G/P_\pm$, and consequently $\dim M = d + r - 2$ and $\dim Y_\pm = d - r$. By \cite[Proposition 48]{bfm_nested} one has the following decompositions of the middle cohomology group of $M$:

    \begin{equation}\label{eq:hodge_decomposition}
        H^{d+r-2}(M, \ZZ) \simeq H^{d-r}(Y_\pm, \ZZ) \oplus H^{d-r+2}(G/P_\pm, \ZZ) \oplus\cdots \oplus H^{d +r - 2}(G/P_\pm, \ZZ)
    \end{equation}
    
    Observe that for $d-r$ odd, all summands coming from  $G/P_\pm$ have odd degree: however, since $G/P_\pm$ is rational homogeneous, its cohomology is purely algebraic, thus concentrated in even degree. Therefore, the decompositions of Equation \ref{eq:hodge_decomposition} yield the following isomorphism:

    \begin{equation}\label{eq:hodge_decomposition_only_Ys}
        H^{d-r}(Y_-, \ZZ)\simeq H^{d+r-2}(M, \ZZ) \simeq H^{d-r}(Y_+, \ZZ)
    \end{equation}

    It is possible to show that Equation \ref{eq:hodge_decomposition_only_Ys} lifts to an isomorphism of \emph{polarized} Hodge structures, i.e. a Hodge-equivalence of $Y_-$ and $Y_+$, by showing that the isomorphisms of Equation \ref{eq:hodge_decomposition_only_Ys} preserve the cup product. This is the content of the following result, which we state without proof as an immediate corollary of \cite[Proposition 3.4]{ourpaper_k3s}:
    
    \begin{proposition}\label{prop:Hodge}
        Let $(Y_-, Y_+)$ be a Calabi--Yau pair of type $A^G_{2n}$, with $n$ odd. Then there exists an isometry of integral Hodge structures:

        \begin{equation}
            H^{n^2-1}(Y_-, \ZZ) \arw H^{n^2-1}(Y_+, \ZZ).
        \end{equation}
    \end{proposition}

    \subsection{Counterexamples to the birational Torelli theorem}
    Putting all together, by Propositions \ref{prop:notiso}, \ref{prop:Leq} and \ref{prop:Hodge} we obtain the main result of this paper, which we state in a self-contained way:

    \begin{theorem}\label{thm:main_body}(Theorem \ref{thm:main_intro})
        For $n\in\NN$, $n\geq2$, consider the locally trivial $\PP^n$-fibrations $p_-:F(n, n+1, 2n+1)  \arw G(n, 2n+1)$ and $p_+:F(n, n+1, 2n+1)  \arw G(n+1, 2n+1)$ and a general section $s\in H^0(F(n, n+1, 2n+1), p_-^*\Oc(1)\otimes p_+^*\Oc(1))$. Let $(Y_-, Y_+)$ be the pair of Calabi--Yau $(n^2-1)$-folds defined as $Y_\pm:=Z(p_{\pm*}s)$. Then:

        \begin{enumerate}
            \item $Y_-$ and $Y_+$ are not birationally equivalent
            \item $([Y_-]-[Y_+])\LL^n = 0$, i.e. $Y_-$ and $Y_+$ are $\LL$-equivalent
            \item For $n$ even, there is a Hodge isometry $H^{n^2-1}(Y_-, \ZZ) \simeq H^{n^2-1}(Y_+, \ZZ)$, i.e. $Y_-$ and $Y_+$ are Hodge equivalent.
        \end{enumerate}
    \end{theorem}

\section{Relation with intersections of translates}

    In \cite{grosspopescu, grzegorz, kanazawa, michal} intersections of two translates of $\grass_+\subset\PP_+$ have been considered, for $n=2$. With this construction, they described a 51 -dimensional family $\Xc_{25}$ of Calabi--Yau threefolds, among which \cite{ottemrennemo} and \cite{borisovcaldararuperry} constructed counterexamples to the birational Torelli theorem. In \cite{ottemrennemo}, a divisor in $\Xc_{25}$ has been described in terms of intersections of ``infinitesimal'' translates, i.e. zero loci of the normal bundle of $\grass_+$ in $\PP$. These are Calabi--Yau threefolds of type $A^G_4$: for $Y\in\Xc_{25}$ one can check that, while the dimension of the moduli space is $\dim \Xc_{25} = h^1(T_Y) = 51$, the number of parameters describing $Y = Z(s)$ for $s\in H^0(\grass_+, \Qc^\vee(2))$ is $h^0(\grass_+, \Qc^\vee(2))- \dim\Aut \grass_+ = 75-24 = 51$, and therefore the $A^G_4$ construction cannot describe the whole family $\Xc_{25}$. While for the general case $V\simeq\CC^{2n+1}$ the intersection of two general translates of $\grass_+$ is empty, it is a natural question to ask, for a Calabi--Yau $(n^2-1)$-fold of type $A^G_{2n}$, whether $h^{1, n^2-1}(Y)$ and $h^0(\grass_+, \Qc^\vee(2))- \dim\Aut \grass_+$ coincide.

    \begin{lemma}\label{lem:dimension_family}
        Let $Y\subset \grass_+$ be a Calabi--Yau $(n^2-1)$-fold of type $A^G_{2n}$. Assume $n>2$. Then $h^{1, n^2-1}(Y) = h^0(\grass_+, \Qc^\vee(2))- \dim\Aut \grass_+-1$.
    \end{lemma}

    \begin{proof}
        By Serre duality $h^{1, n^2-1}(Y) = h^{n^2-1}(Y, \Omega^1_Y) = h^1(Y, T_Y)$. Then, by the normal bundle sequence, we need to compute the cohomology of $T_{\grass_+}|_Y \simeq \Uc^\vee\otimes\Qc|_Y$ and $\Qc^\vee(2)|_Y\simeq\wedge^n\Qc(1)|_Y$.\\
        Let us begin wth the latter. The tensor product of the Koszul resolution \ref{eq:koszul} with $\wedge^n\Qc(1)$ is given by composing the long exact sequence \ref{eq:tensored_koszul_1} by the short exact sequence defining the ideal sheaf $\Ic_{Y|\grass_+}$. The $l$-th term of such resolution is $\wedge^{l}\Qc\otimes\wedge^n\Qc(-2l + 1)$, and for $l>1$ it has no cohomology by Lemma \ref{lem:main_vanishing}. For $l = 1$ one has $\Qc\otimes\wedge^n\Qc(-1) \simeq \SS^{(2, 1^{n-1},0)}\Qc(-1)\oplus\Oc$ where the second summand has cohomology concentrated in degree zero and dimension one, while the first summand does not contribute (again by Lemma \ref{lem:main_vanishing}). For $l = 0$ the bundle has cohomology concentrated in degree zero. Thus, we deduce that $\Qc^\vee(2)|_Y$ has no cohomology in higher degree, while $h^0(Y, \Qc^\vee(2)|_Y) = h^0(\grass_+, \Qc^\vee(2)) -1$.\\
        Consider now $\Uc^\vee\otimes\Qc|_Y$. This object, similarly to the former, has a Koszul resolution whose $l$-th term is $\Uc^\vee\otimes\Qc\otimes\wedge^l\Qc(-2l)$. By tensoring it with the tautological short exact sequence (Equation \ref{eq:tautological_sequence}), we obtain the following resolution:

        \begin{equation*}
            0\arw \wedge^n\Qc\otimes\Qc\otimes\wedge^l\Qc(-2l-1) \arw V^\vee\otimes\Qc\otimes\wedge^l\Qc(-2l) \arw  \Uc^\vee\otimes\Qc\otimes\wedge^l\Qc(-2l) \arw 0.
        \end{equation*}

        %Here, for $1\leq l \leq n-1$ the first two bundles (and hence the third one) have no cohomology by Lemma \ref{vanishing_dimension_family}.
        Consider the case $l = 0$: there it is known that $h^0(\grass_+, \Uc^\vee\otimes\Qc)$ is the dimension of the adjoint representation of $SL(V)$, which is $\dim\Aut\grass_+$. For higher values of $l$ the computations are covered in Lemma \ref{lem:vanishings_dimension_family}, where we conclude that the only contributing term is for $l = n+1$. However, while such term gives $H^{n^2-n-1}(\grass, \Uc^\vee\otimes\Qc\otimes\wedge^{n-+}\Qc(-2n-2) = \CC$, such cohomology does not contribute to $H^1(Y, T_Y)$ because the degree is too high compared to the length of the Koszul resolution.\\
        Summing all up, we find $h^{1, n^2-1}(Y) = h^0(\grass_+, \Qc^\vee(2)) - \dim\Aut \grass_+-1$, thus the proof is concluded.
    \end{proof}

    \begin{remark}
        One might observe that, doing the same computation for $n = 2$, a difference appears in the Koszul resolution of $\Uc^\vee\otimes\Qc|_Y$: in fact, while for $n>2$ the term $\Uc^\vee\otimes\Qc\otimes\Qc(-2)\simeq \Uc^\vee\otimes\wedge^2\Qc(-2)\oplus \Uc^\vee\otimes\Sym^2\Qc(-2)$ has no cohomology, for $n = 2$ its first direct summand has a non-trivial (one dimensional) $H^1$, which is responsible to the difference in the dimension count of Lemma \ref{lem:dimension_family}, and thus to the fact that the construction as varieties of type $A^G_4$ cannot describe the whole family $\Xc_{25}$.
    \end{remark}

\appendix

\section{Borel--Weil--Bott computations}\label{appendix_bott}

We will gather here all technical statements which come as applications of the Borel--Weil--Bott theorem \cite{bott}, which are needed in the previous sections. Let us begin by recalling the statement of the theorem.

\begin{theorem}[Borel--Weil--Bott]\label{thm:borel_weil_bott}
        Let $\Ec_\omega$ be a homogeneous, irreducible vector bundle on a rational homogeneous variety $G/P$, let us call $\rho$ the sum of all fundamental weights. Then one and only one of the following statements is true:
        \begin{enumerate}
            \item there exists a sequence $s_p$ of simple Weyl reflections of length $p$ such that $s_p(\omega+\rho) - \rho$ is dominant (i.e. all coefficients of its the expansion in the fundamental weights are non-negative). Then $H^p(G/P, \Ec_\omega)\simeq \VV^G_{s_p(\omega+\rho) - \rho}$ and all the other cohomology is trivial.
            \item  there is no such $s_p$ as in point (1). Then $\Ec_\omega$ has no cohomology.
        \end{enumerate}
    \end{theorem}

    In the case $G/P = SL(V)/P_k \simeq G(k, V)$, which is the only one we are interested in, the Weyl reflection associated to the simple root $\alpha_i$ acts as follows:

    \begin{equation}
        s_{\alpha_i}(\omega)_j = \left\{
        \begin{array}{cc}
            \omega_j & j\notin\{i-1; i; i+1\} \\
            \omega_j-\omega_i & j\in\{i-1; i+1\} \\
            -\omega_i & j = i.
        \end{array}
        \right.
    \end{equation}

    \begin{remark}
        Note that an alternative, simpler formulation of the Borel--Weil--Bott theorem for Grassmannians is described in \cite[Appendix A]{borisovcaldararuperry}. However, for the purpose of the computations below, we choose the general formulation. In fact, the combinatorial aspects of the proof of Lemma \ref{lem:main_vanishing} turn out to be far more readable in this language.
    \end{remark}
    
    Let us switch again to the shorthand notation $\grass:= G(n, V)$ for $V\simeq \CC^{2n+1}$, and assume $n>2$. All the cohomological vanishings needed in this paper boil down to the content of the following lemma, which we state slightly more generally:
    
    \begin{lemma}\label{lem:main_vanishing}
        For $0\leq \lambda_i\leq n-1$ and $0 < i < 2n+1$ one has:
    
        \begin{equation*}
            \begin{split}
                H^\bullet(\grass, \SS^{(\lambda_1, \dots, \lambda_{n+1})}\Qc(-i))  &= 0 %\\
                %H^\bullet(\grass, \SS^{(\mu_1, \dots, \mu_n)}\Uc^\vee(-i))  &= 0
            \end{split}
        \end{equation*}
    \end{lemma}
    
    \begin{proof}
        By the discussion of Section \ref{subsec:grassmannians} the weight associated to $\SS^{(\lambda_1, \dots, \lambda_{n+1})}\Qc(-i)$ has the form:
    
        \begin{equation*}
            \begin{array}{cccccccccc}
                \omega & =( 0^{n-1}, &-i, &j_1 & \dots & j_n)
            \end{array}
        \end{equation*}
        
        where $\sum_k j_k \leq n-1$. Adding the sum of fundamental weights yields:
    
        \begin{equation*}
            \begin{array}{cccccccccc}
                \omega+\rho & =( 1^{n-1}, &1-i, &1+j_1 & \dots & 1+j_n)
            \end{array}
        \end{equation*}
    
        Let us now perform, repeatedly, the operation of applying the Weyl reflection which changes sign of the leftmost negative entry. If $i\leq n$, we will eventually find a zero in one of the first $n-1$ entries. Then, applying again the same algorithm, we will find a weight with no negative entries and at least a zero, which by Theorem \ref{thm:borel_weil_bott} implies that $\SS^{(\lambda_1, \dots, \lambda_{n+1})}\Qc(-i)$ has no cohomology. 
        
        Otherwise, if $i>n$, after exactly $n$ repetitions of the operation described above, we find the weight:
    
        \begin{equation*}
            \begin{array}{cccccccccc}
                w_1\circ\cdots\circ w_n (\omega+\rho) & =(i-n, & 1^{n-1}, &2+j_1-i, &1+ j_2 & \dots & 1+j_n)
            \end{array}
        \end{equation*}
    
        Recall that our assumptions on $(\lambda_1,\dots,\lambda_n)$ imply $j_1<n$: since $i>n$, it follows that $2+j_1-i$ cannot be positive. If it is negative, we can repeat the same procedure starting by applying $w_{n+1}$: if $n-1+j_1-i\leq 0$ we hit again a zero in the one of the first $n$ entries and we proceed as above, otherwise we end up with the weight:
    
        \begin{equation*}
            \begin{array}{cccccccccc}
                \prod_{t = 1}^{n+1} w_t\circ\prod_{t=1}^n w_t (\omega+\rho) & =(j_1+i-n, & 1^n, &3+j_1+j_2-i, &1+ j_3 & \dots & 1+j_n)
            \end{array}
        \end{equation*}
    
        where again $3+j_1+j_2 -i$ is negative. This algorithm can be repeated until we either find a weight with no negative entries and at least one zero (and therefore no cohomology), or in the case $n+1+\sum_t j_t -i <0$ we end up with the weight
    
        \begin{equation*}
            \begin{array}{cccccccccc}
               \prod_{t = 1}^{2n-1} w_t\circ \cdots \circ\prod_{t=1}^n w_t (\omega+\rho) & =(\sum j_t+i-n, & 1^{2n-2}, & n+1+\sum j_t -i).
            \end{array}
        \end{equation*}
    
        Here, if $\sum j_t + 3n - 1 - i < 0$, we might get nonvanishing cohomology in degree $\dim G = n^2 + n$: however, this cannot happen for the assumption $0<i<2n+1$. Therefore, we conclude that $\SS^{(\lambda_1, \dots, \lambda_{n+1})}\Qc(-i)$ has no cohomology.%\\
        \begin{comment}
        \\
        The second part of the statement follows by exactly the same reasoning. In fact, the weight associated to $\SS^{(\mu_1, \dots, \mu_n)}\Uc^\vee(-i)$ has the shape:

        \begin{equation*}
            \begin{array}{cccccccccc}
                \omega & =( l_1, & \dots & l_{n-1} &-i, &, 0^n)
            \end{array}
        \end{equation*}

        where $\sum_k l_k \leq n-1$, which, up to reversing the order of the entries, is nearly identical to the expression we found for $\SS^{(\lambda_1, \dots, \lambda_{n+1})}\Qc(-i)$. 
        \end{comment}
    \end{proof} 

    \begin{lemma}\label{lem:q_is_stable_vanishings}
        For $0\leq l\leq n+1$, the bundle $\wedge^k\Qc\otimes\wedge^l\Qc(-1-2l)$ has no cohomology in degree $p<n+1$.
    \end{lemma}
    
    \begin{proof}
        By the Littlewood--Richardson formula, the bundle $\wedge^k\Qc\otimes\wedge^l\Qc(-1-2l)$ decomposes in summands which are of the form we treated in Lemma \ref{lem:main_vanishing}, and hence none of them contributes to the cohomology.
    \end{proof}
    
    \begin{lemma}\label{lem:vanishings_hom_is_surjectve}
        For $0\leq k \leq n+1$ and $p<n+1$ one has:
        
        \begin{equation*}
            H^p(\grass, \Qc^\vee(2)\otimes\wedge^k\Qc(-2k)) = 0.
        \end{equation*}
    \end{lemma}
    
    \begin{proof}
        One has, for $0\leq k \leq n+1$:
    
        \begin{equation}\label{eq:lr_computation}
            \begin{split}
                H^p(\grass, \Qc^\vee(2)\otimes\wedge^k\Qc(-2k)) & = H^p(\grass, \wedge^n\Qc\otimes\wedge^k\Qc(-2k+1)) \\
                & = H^p(\grass, \SS^{(2^k, 1^{n-k})}\Qc(-2k+1))\oplus H^p(\grass, \wedge^{k-1}\Qc(-2k+2)). \\
            \end{split}
        \end{equation}
    
        Both summands are of the type discussed in Lemma \ref{lem:main_vanishing}, and therefore they have no cohomology.
    \end{proof}
    
    \begin{lemma}\label{lem:spaces_of_sections}
    Consider $Y = Z(s)$, where $s\in H^0(\grass, \Qc_{\grass}^\vee(2))$ is general. Then one has:
    
    \begin{equation*}
        \begin{split}
            H^0(Y, \Oc(1)) & \simeq H^0(\PP, \Oc(1)) \\
            H^0(Y, \Qc_{\grass}^\vee(2)|_Y) & \simeq H^0(\grass, \Qc_{\grass}^\vee(2)) 
        \end{split}
    \end{equation*}
    \end{lemma}
    
    \begin{proof}
        Since $\PP := \PP(\wedge^n V)$ one has $H^0(\PP, \Oc(1)) = \wedge^n V^\vee$, and by the Borel--Weil--Bott theorem one easily shows that $H^0(\grass, \Oc(1)) = \wedge^n V^\vee$, hence the two spaces are equal. To show that $H^0(\grass, \Oc(1)) = H^0(Y, \Oc(1))$ we take the tensor product of the Koszul exact sequence \ref{eq:koszul} by $\Oc(1)$. By Lemma \ref{lem:main_vanishing} all terms have no cohomology except for $\Oc(1)$ and $\Oc_Y(1)$, and this settles the first claim.\\
        To address the second one, this time we tensor the sequence \ref{eq:koszul} by $\Qc$: all terms but $\Qc$ and $\Qc|_Y$ are of the form we discussed in Lemma \ref{lem:main_vanishing}, and therefore they have no cohomology. 
    \end{proof}

    \begin{lemma}\label{lem:vanishings_dimension_family}
        For $1\leq l\leq n+1$ one has $H^\bullet(\grass, \wedge^n\Qc\otimes\Qc\otimes\wedge^l\Qc(-2l-1) ) = H^\bullet(\grass, \Qc\otimes\wedge^l\Qc(-2l)) = 0$, except for $l = n+1$ where $H^{n^2-n}(\grass, \wedge^n\Qc\otimes\Qc\otimes\wedge^l\Qc(-2l-1) ) = \CC$.
    \end{lemma}

    \begin{proof}
        For $l<n$ by an expansion similar to Equation \ref{eq:lr_computation} one can use the Littlewood--Richardson formula to write both the bundles as direct sums of objects of the form we already treated in Lemma \ref{lem:main_vanishing}, therefore the proof reduces to considering the cases $l = n$ and $l = n+1$. We start by considering both bundles for $l = n$:

        \begin{equation*}
            \wedge^n\Qc\otimes\Qc\otimes\wedge^n\Qc(-2n-1) = \SS^{(3,2^{n-1},0)}\Qc(-2n-1) \oplus \Ec
        \end{equation*}

        where $\Ec$ is again a sum of terms which do not contribute by Lemma \ref{lem:main_vanishing}. The easiest way to show that the first summand has no cohomology is to use Serre duality:

        \begin{equation*}
            H^p(\grass, \SS^{(3,2^{n-1},0)}\Qc(-2n-1)) \simeq H^{n^2-n-p}(\grass, \SS^{(3,2^{n-1},0)}\Qc^\vee) \simeq H^{n^2-n-p}(\grass, \SS^{(3,1^{n-1},0)}\Qc^\vee(-3)),
        \end{equation*}

        and to conclude by Lemma \ref{lem:main_vanishing}. The same exact steps prove that $\Qc\otimes\wedge^n\Qc(-2n)$ has no cohomology.\\
        Fix now $l = n+1$: the first bundle is 

        \begin{equation*}
            \wedge^n\Qc\otimes\Qc\otimes\wedge^{n+1}\Qc(-2n-3) = \wedge^n\Qc\otimes\Qc(-2n-2) = \Oc(-2n-1)\oplus\SS^{(2, 1^{n-1}, 0)}\Qc(-2n-2).
        \end{equation*}

        Here the first summand contributes with $H^{n^2-n}(\grass, \Oc(-2n-1)) = H^0(\grass, \Oc) \simeq \CC$. The second summand, by the same kind of argument, gives no contribution, and the same happens for $\Qc\otimes\wedge^{n+1}\Qc(-2n-2) \simeq \Qc(-2n - 1)$. 
    \end{proof}

\bibliographystyle{alpha}
\bibliography{bibliography}

\begin{thebibliography}{JMPSE17}

\bibitem[Arr96]{arrondo}
Enrique Arrondo.
\newblock {Subvarieties of Grassmannians}.
\newblock {\em Lecture notes}, 1996.

\bibitem[BCP18]{borisovcaldararuperry}
Lev Borisov, Andrei Căldăraru, and Alexander Perry.
\newblock Intersections of two grassmannians in $\mathbb{P}^9$.
\newblock {\em Journal f{\"u}r die reine und angewandte Mathematik (Crelles
  Journal)}, 2020:133 -- 162, 2018.

\bibitem[BFM21]{bfm_nested}
Marcello Bernardara, Enrico Fatighenti, and Laurent Manivel.
\newblock {Nested varieties of {K3} type}.
\newblock {\em Journal de l{\textquoteright}\'Ecole polytechnique {\textemdash}
  Math\'ematiques}, 8:733--778, 2021.

\bibitem[Bot57]{bott}
Raoul Bott.
\newblock {Homogeneous Vector Bundles}.
\newblock {\em Annals of Mathematics}, 66(2):203--248, 1957.

\bibitem[FKMR21]{ourpaper_generalizedroofs}
Enrico Fatighenti, Michał Kapustka, Giovanni Mongardi, and Marco Rampazzo.
\newblock {The generalized roof F(1,2,n): Hodge structures and derived
  categories}.
\newblock {\em to appear in Algebras and Representation Theory}, 2021.

\bibitem[GP01]{grosspopescu}
Mark Gross and Sorin Popescu.
\newblock {Calabi–Yau Threefolds and Moduli of Abelian Surfaces I}.
\newblock {\em Compositio Mathematica}, 127(2):169–228, 2001.

\bibitem[Huy16]{huybrechts_k3s}
Daniel Huybrechts.
\newblock {\em Lectures on K3 Surfaces}.
\newblock Cambridge Studies in Advanced Mathematics. Cambridge University
  Press, 2016.

\bibitem[IMOU19]{imou_G2}
Atsushi {Ito}, Makoto {Miura}, Shinnosuke {Okawa}, and Kazushi {Ueda}.
\newblock {The class of the affine line is a zero divisor in the Grothendieck
  ring: via $G_2$-Grassmannians}.
\newblock {\em J. Algebraic Geom.}, 28:245--250, dec 2019.

\bibitem[JMPSE17]{jardimmenetprataearp}
Marcos Jardim, Grégoire Menet, Daniela~M. Prata, and Henrique~N. Sá~Earp.
\newblock Holomorphic bundles for higher dimensional gauge theory.
\newblock {\em Bulletin of the London Mathematical Society}, 49(1):117--132,
  2017.

\bibitem[{Kan}10]{kanazawa}
Atsushi {Kanazawa}.
\newblock {Pfaffian Calabi-Yau Threefolds and Mirror Symmetry}.
\newblock {\em arXiv e-prints}, page arXiv:1006.0223, June 2010.

\bibitem[Kan22]{kanemitsu}
Akihiro Kanemitsu.
\newblock Mukai pairs and simple k-equivalence.
\newblock {\em Mathematische Zeitschrift}, 2022.

\bibitem[Kap85]{kapranov_grassmannians}
M.~M. Kapranov.
\newblock On the derived category of coherent sheaves on {Grassmann} manifolds.
\newblock {\em Math. USSR, Izv.}, 24:183--192, 1985.

\bibitem[Kap08]{grzegorz}
Grzegorz Kapustka.
\newblock {Primitive contractions of Calabi–Yau threefolds II}.
\newblock {\em Journal of the London Mathematical Society}, 79(1):259--271, 12
  2008.

\bibitem[Kap11]{michal}
Michal Kapustka.
\newblock {Geometric transitions between Calabi-Yau threefolds related to
  Kustin-Miller unprojections}.
\newblock {\em J. Geom. Phys.}, 61:1309--1318, 2011.

\bibitem[KR19]{ourpaper_cy3s}
Micha\l{} Kapustka and Marco Rampazzo.
\newblock {Torelli problem for Calabi\textendash{}Yau threefolds with GLSM
  description}.
\newblock {\em Commun. Num. Theor. Phys.}, 13(4):725--761, 2019.

\bibitem[KR22]{ourpaper_k3s}
Micha\l\ Kapustka and Marco Rampazzo.
\newblock Mukai duality via roofs of projective bundles.
\newblock {\em Bulletin of the London Mathematical Society}, 54(2):694--717,
  2022.

\bibitem[KS16]{kuznetsovshinder}
Alexander Kuznetsov and Evgeny Shinder.
\newblock Grothendieck ring of varieties, d- and l-equivalence, and families of
  quadrics.
\newblock {\em Selecta Mathematica}, 24:3475--3500, 2016.

\bibitem[Man19]{manivel}
Laurent Manivel.
\newblock {Double spinor Calabi-Yau varieties}.
\newblock {\em {Épijournal de Géométrie Algébrique}}, {Volume 3}, 2019.

\bibitem[OR18]{ottemrennemo}
John~Christian Ottem and J{\o}rgen~Vold Rennemo.
\newblock A counterexample to the birational torelli problem for calabi–yau
  threefolds.
\newblock {\em Journal of the London Mathematical Society}, 97, 2018.

\bibitem[Ram20]{mypaper_roofbundles}
Marco Rampazzo.
\newblock {Calabi-Yau fibrations, simple K-equivalence and mutations}.
\newblock {\em ArXiv preprint}, 2020.

\bibitem[Ume78]{umemura_stability}
Hiroshi Umemura.
\newblock {On a theorem of Ramanan}.
\newblock {\em Nagoya Mathematical Journal}, 69:131–138, 1978.

\bibitem[Ver13]{verbitsky_hks}
Misha Verbitsky.
\newblock {Mapping class group and a global Torelli theorem for hyperkähler
  manifolds}.
\newblock {\em Duke Mathematical Journal}, 162(15):2929 -- 2986, 2013.

\bibitem[Wey03]{weyman}
Jerzy Weyman.
\newblock {\em {Cohomology of Vector Bundles and Syzygies}}.
\newblock Cambridge Tracts in Mathematics. Cambridge University Press, 2003.

\end{thebibliography}

\end{document}